\documentclass[leqno,twoside,11pt]{amsart}

\usepackage{latexsym, esint, url, amssymb, yfonts}
\usepackage{graphicx,color}

\theoremstyle{plain}

\def\endproof{\hspace*{\fill}\mbox{\ \rule{.1in}{.1in}}\medskip }

\newtheorem{theorem}{Theorem}[section]
\newtheorem{lemma}[theorem]{Lemma}

\newtheorem{definition}[theorem]{Definition}

\theoremstyle{definition}

\numberwithin{equation}{section}

\newcommand{\R}{\mathbb{R}}

\begin{document}

\title[Double-obstacle problem and differential games]
{Discrete approximations to the double-obstacle problem, and
  optimal stopping of  tug-of-war games}
\author{Luca Codenotti}
\author{Marta Lewicka}
\author{Juan Manfredi}
\address{University of Pittsburgh, Department of Mathematics,
301 Thackeray Hall, Pittsburgh, PA 15260, USA }
\email{luc23@pitt.edu, lewicka@pitt.edu, manfredi@pitt.edu }

\date{November 1, 2015}

\begin{abstract}
We study the double-obstacle problem for the $p$-Laplace operator,
$p\in [2, \infty)$. We prove that for Lipschitz boundary data and Lipschitz obstacles,
viscosity solutions are unique and coincide with  variational
solutions. They are also uniform limits of solutions to discrete
min-max problems that can be interpreted as the dynamic programming
principle for appropriate tug-of-war games with noise. In these games,
both players in addition to choosing their strategies, are also
allowed to choose stopping times. The solutions to the
double-obstacle problems are limits of values of these games, 
when the step-size controlling the single shift in the
token's position, converges to $0$. We propose a numerical scheme
based on this observation and show how it works for some 
examples of obstacles and boundary data.
\end{abstract}

\maketitle

\section{Introduction}

The purpose of this paper is to study the double obstacle problem for
the $p$-Laplace operator:
\begin{equation}\label{plap}
-\Delta_p u = -\mbox{div}\big(|\nabla u|^{p-2} \nabla u\big), \qquad p\in [2, \infty).
\end{equation}
Let $\Omega\subset\R^N$ be an open, bounded domain with Lipschitz
boundary and let $F:\partial\Omega\rightarrow\R$  be a Lipschitz
continuous boundary datum. Given are bounded and Lipschitz functions
$\Psi_1, \Psi_2:\R^N\rightarrow\R$  such that $\Psi_1 \leq \Psi_2$ in $\bar{\Omega}$ and
$\Psi_1\leq F \leq \Psi_2$ on $\partial\Omega$. We interpret $\Psi_1$
and $\Psi_2$ as the lower and upper obstacles, respectively, and
consider the following double-obstacle problem:
\begin{equation}\label{obst}
\begin{cases}
-\Delta_pu\geq 0 & \mbox{ in }\{x\in\Omega; ~ u(x) < \Psi_2(x)\}\\
-\Delta_pu\leq 0 & \mbox{ in }\{x\in\Omega; ~ u(x) > \Psi_1(x)\}\\
\Psi_1\leq u \leq\Psi_2 & \mbox{ in }\Omega\\
u=F & \mbox{ on }\partial\Omega.
\end{cases}
\end{equation}

Note that under the third condition in (\ref{obst}), the first two
conditions are jointly equivalent to:
$$\min\Big\{\Psi_2 - u, \max\big\{\Delta_p u, \Psi_1 - u\big\}\Big\} =0.$$
That is, when $u$ does not coincide with $\Psi_1$ we require it to be a subsolution, and likewise it must
be a supersolution when it does not coincide with $\Psi_2$. In
particular, $u$ must actually be $p$-harmonic outside of the contact sets with both obstacles:
$$-\Delta_p u = 0 \quad \mbox{ in } \{x\in\Omega; ~ \Psi_1(x) < u(x) < \Psi_2(x)\}.$$

\medskip

\begin{definition} \label{viscsol}
We say that a continuous function $u:\bar\Omega\to\R$ is a viscosity
solution of the double-obstacle problem (\ref{obst}), when:
\begin{itemize}
\item[{(i)}] $u=F$ on $\partial\Omega$ and $\Psi_1\leq u\leq \Psi_2$ in
  $\Omega$.
\item[{(ii)}] For every $x_0\in\Omega$ such that $u(x_0) > \Psi_1(x_0)$
  and every $\phi \in \mathcal{C}^2(\Omega)$ such that:
$$\phi(x_0)= u(x_0), \quad \phi>u \mbox{ in } {\Omega}\setminus\{x_0\}, \quad \nabla\phi(x_0)\neq 0$$
there holds: $\Delta_p \phi(x_0)\geq 0$.
\item[{(iii)}] For every $x_0\in\Omega$ such that $u(x_0) < \Psi_2(x_0)$
  and every $\phi \in \mathcal{C}^2(\Omega)$ such that:
$$\phi(x_0)= u(x_0), \quad \phi<u \mbox{ in } {\Omega}\setminus\{x_0\}, \quad \nabla\phi(x_0)\neq 0$$
there holds: $\Delta_p \phi(x_0)\leq 0$.
\end{itemize}
\end{definition}

\medskip

Our first result concerns existence and uniqueness of
solutions to the min-max problem that, as we shall see, can serve as
a uniform approximation of the original problem (\ref{obst}) in the
sense that its solutions converge uniformly to the viscosity solution
of Definition \ref{viscsol}.
Let $0<\bar\epsilon_0\ll 1$ be a small constant and define the sets:
$$\Gamma=\{x\in\R^N\setminus\Omega;~ \mbox{dist}(x,\Omega)<\bar\epsilon_0\},\qquad
X=\Gamma\cup\Omega.$$

\begin{theorem}\label{epsilonp}
Let $\alpha\in [0,1)$ and $\beta = 1-\alpha$. Let $\Psi_1, \Psi_2:\R^N\rightarrow\R$
and $F:{\Gamma}\rightarrow\R$ be bounded Borel functions such that $\Psi_1 \leq \Psi_2$ in $X$ and
$\Psi_1 \leq F \leq \Psi_2$ in $\Gamma$. Then, for every $\epsilon<\bar\epsilon_0$, 
there exists a unique Borel function $u_\epsilon: X\to \R$ which satisfies: 
\begin{equation}\label{minmax}
u_\epsilon(x) = 
\begin{cases}
\max\Bigg\{\Psi_1(x),\min\Big\{\displaystyle{\Psi_2(x),\frac{\alpha}{2}\sup_{B_\epsilon(x)}
u_\epsilon + \frac{\alpha}{2}\inf_{B_\epsilon(x)} u_\epsilon +
\beta\fint_{B_\epsilon(x)}u_\epsilon}\Big\}\Bigg\} &   \mbox{for }  x\in\Omega, \\
F(x) & \mbox{for }  x\in\Gamma.
\end{cases}
\end{equation}
\end{theorem}

We now state our main result:

\begin{theorem} \label{main}
Let $p\in [2, \infty)$ and define: 
$$\displaystyle{\alpha=\frac{p-2}{p+N}}, \qquad \displaystyle{\beta=\frac{2+N}{p+N}}.$$  
Let $F, \Psi_1, \Psi_2:\R^N\rightarrow\R$ be bounded Lipschitz continuous
functions such that:
$$\Psi_1 \leq \Psi_2 \quad \mbox{ in } \bar\Omega \qquad \mbox{ and } \qquad
\Psi_1 \leq F \leq \Psi_2 \quad \mbox{ in } \mathbb{R}^N\setminus\Omega.$$
Let $u_\epsilon$ be the unique solution to (\ref{minmax}). Then $\{u_\epsilon\}$ converge, as $\epsilon \to 0$,
uniformly in $\bar\Omega$, to the continuous function $u:\bar\Omega\to\R$ which is the
unique viscosity solution to the double-obstacle problem (\ref{obst}).
\end{theorem}

Clearly, the above limit $u$ depends only on the values of $F$ on $\partial\Omega$ 
and values of $\Psi_1, \Psi_2$ in $\bar\Omega$, and therefore
any Lipschitz continuous extensions of $F, \Psi_1, \Psi_2$ on $\R^N$
(which exist by virtue of Kirszbraun's extension theorem) give the same limit.

\bigskip

Theorem \ref{epsilonp} and Theorem \ref{main} will be proved in
sections \ref{sec2} and \ref{sec3}, whereas uniqueness of
viscosity solutions to (\ref{obst}) will be proved in section
\ref{secuni}. In section \ref{sec4} we show that (\ref{minmax}) can be seen as the dynamic programming
principle for a stochastic-deterministic tug-of-war game, where the
two players are allowed to choose their strategies as well as
stopping times. The connection between tug-of-war games and the
nonlinear operator $\Delta_p$ stems from the fact that, for a
sufficiently regular $u$ one can express its $p$-Laplacian as a
combination the $\infty$-Laplacian and the ordinary Laplacian:
\begin{equation*}
\Delta_p u (x)= |\nabla u|^{p-2}\big((p-2)\Delta_\infty u(x) + \Delta u (x)\big),
\end{equation*}
where:
\begin{equation*}
\Delta_\infty u(x) = \Big\langle \nabla^2 u(x) \frac{\nabla u(x)}{|\nabla
u(x)|}, \frac{\nabla u(x)}{|\nabla u(x)|}\Big\rangle. 
\end{equation*}
The tug-of-war interpretation of the $\infty$-Laplacian has been
developed in the fundamental paper \cite{PSSW}, while it is well known
that the values of the discrete Brownian motion converge to a harmonic
function. Thus, an appropriate \lq\lq mixture\rq\rq of the two processes (via the parameters
$\alpha$ and $\beta$) yields $p$-harmonic functions in the limit 
as the discrete step-size $\epsilon\to 0$.

The single obstacle problem for $\Delta_\infty$ has been studied, from
this point of view, in \cite{MRS}. The case $p\in [2,\infty)$, still
in presence of the single obstacle, has been derived in \cite{LM}.
Let us also note that existence, uniqueness and regularity of solutions to the
double-obstacle problem for $\Delta_\infty$ in the domain $\Omega=\R^N$ have been achieved, under
additional assumptions on the Lipschitz obstacles $\Psi_1,
\Psi_2$, in \cite[Theorems 5.1 and 5.2]{BCF} using  barrier
methods. In the same paper, the authors give a heuristic connection to a general non-local
variant of the tug-of-war game. \par

The existence and uniqueness of solutions to double obstacle problems
for convex functionals follows from convex analysis in a standard
way. Questions of regularity of solutions, interior and at the
boundary, have been studied in \cite{DMV} for the linear case and in
\cite{KZ} for the quasilinear case.  Let us point out that there is a
monotonicity property that holds naturally in the single obstacle
problem, namely the solution can be expressed as a supremum of
sub-solutions (or a infimum of super-solutions), that does not hold
in the double obstacle case.  Certain aspects of the regularity proof
in \cite{KZ} are very different in the double obstacle case from the
parallel argument in the single obstacle case. Similarly, our arguments are based on,  but
quite different in the details from,  the arguments in the single obstacle
case \cite{LM}.  In particular, we follow the modern exposition of 
Farnana \cite{farn} for the classical variational theory, which is
valid in general metric measure spaces, and prove that \lq\lq
viscosity = weak\rq\rq\ for double obstacle problems in section
\ref{secuni}.

Finally, in section \ref{secnum} we present examples of numerical
calculations using an   algorithm based on Theorem \ref{epsilonp}  and
Theorem \ref{main}. A numerical algorithm for solving the
double-obstacle problem has been 
proposed in \cite{WC}, where the coincidence set is approximated by
consecutive iterations.  A different algorithm,
taking advantage of the parabolic pde: $u_t - \Delta_2u = 0$ has been
indicated in \cite{RM}. Finite difference methods for the $\infty$
and $p$-laplacian were considered in \cite{Ober}.

\bigskip

\noindent{\bf Acknowledgments.}
M.L. was partially supported by  NSF award DMS-1406730. 

\section{The discrete approximation: a proof of Theorem \ref{epsilonp}}\label{sec2}

The proof relies on Perron's method and it is the same as in
\cite{LPS, LM}.

\smallskip

{\bf 1.} For any bounded Borel function $v:X\rightarrow\R$ we set:
\begin{equation}\label{2.2}
Tv(x) = 
\begin{cases}
\max\Big\{\Psi_1(x),\min\big\{\displaystyle{\Psi_2(x),\frac{\alpha}{2}\sup\limits_{B_\epsilon(x)}
v + \frac{\alpha}{2}\inf\limits_{B_\epsilon(x)} v +
\beta\fint_{B_\epsilon(x)}v}\big\}\Big\} &\mbox{for } x\in\Omega, \\
F(x) &\mbox{for } x\in\Gamma. 
\end{cases}
\end{equation}
It is easy to see that if $v\leq w$ in  $X$ then $Tv \leq Tw$.
Define recursively the sequence of Borel functions $\{u_n\}_{n=1}^\infty$ by:
$$u_0=\chi_{\Gamma}F + \chi_{\Omega}\Psi_1 \qquad \mbox{ and } \qquad
u_{n+1}=Tu_n \quad \forall n\geq 0.$$
We note that $u_0 \leq u_1$, as by construction $\Psi_1\leq T\Psi_1$
in $\Omega$. Consequently, $\{u_n\}$ is pointwise non-decreasing. On
the other hand, it follows from (\ref{minmax}) that $\Psi_1\leq
u_n\leq \Psi_2$ in $\Omega$  and $F =u_n$ in $\Gamma$. Thus, $\{u_n\}$ 
pointwise converges to a Borel function $u:X\to\R$ satisfying:
$$\Psi_1 \leq u \leq \Psi_2 \quad \mbox{ in } \Omega \qquad \mbox{
  and } \qquad u=F \quad
\mbox{ in } \Gamma.$$
 
\smallskip

{\bf 2.} We now show that $\{u_n\}$ converges to $u$ uniformly in
$X$. Assume by contradiction that:
$$M=\lim\limits_{n\rightarrow\infty}\sup_{X}~(u-u_n)>0.$$  
Fix a small parameter $\delta>0$ and take $n>1$ such that: 
\begin{equation*}
\sup_{X}~ (u-u_n) < M+\delta \qquad \mbox{and} \qquad
\forall x\in \Omega \quad 
\beta\fint_{B_\epsilon(x)}(u-u_n)\leq\frac{\beta}{|B_\epsilon(x)|}\int_{X}(u-u_n)<\delta,
\end{equation*}
where the monotone convergence theorem guarantees validity of  the second condition above.

Let $x_0\in \Omega$ satisfy: $u(x_0)-u_{n+1}(x_0)>M-\delta>0$. Note
that if $u(x_0)=\Psi_1(x_0)$, then it must be 
$u_n(x_0)=\Psi_1(x_0)=u(x_0)$ for all $n$. Similarly, if $u_{n+1}(x_0)=\Psi_2(x_0)$, it must
be $u_m(x_0)=\Psi_2(x_0)=u(x_0)$ for all $m>n+1$. Therefore:
\begin{equation}\label{pomoc}
\Psi_1(x_0)<u(x_0) \quad \mbox{ and } \quad u_{n+1}(x_0)<\Psi_2(x_0).
\end{equation}
Choose $m>n$ such that $u_{m+1}(x_0)-u_{n+1}(x_0)>M-2\delta$ and $u_{m}(x_0)>\Psi_1(x_0)$. 
We  now compute: 
\begin{equation}\label{ineqCont1}
\begin{split}
M-2\delta &<u_{m+1}(x_0)-u_{n+1}(x_0) \\ & \leq
\frac{\alpha}{2}\big(\sup_{B_\epsilon(x_0)}u_{m}-\sup_{B_\epsilon(x_0)}u_{n}\big)
+ \frac{\alpha}{2}\big(\inf_{B_\epsilon(x_0)}u_{m}-\inf_{B_\epsilon(x_0)}u_{n}\big)
+ \beta\fint_{B_\epsilon(x_0)}(u_m-u_n)\\ 
& \leq\alpha \sup_{B_\epsilon(x_0)}(u_m-u_n)+
\beta\fint_{B_\epsilon(x_0)}(u_m-u_n) \leq\alpha \sup_{B_\epsilon(x_0)}(u-u_n)+ 
\beta\fint_{B_\epsilon(x_0)}(u-u_n)\\ & < \alpha(M+\delta)+\delta,
\end{split}
\end{equation}
where in the second inequality we used (\ref{2.2}) and (\ref{pomoc}),
while for the third inequality we noted that both quantities:
$\sup_{B_\epsilon(x_0)}u_{m}-\sup_{B_\epsilon(x_0)}u_{n}$ and:
$\inf_{B_\epsilon(x_0)}u_{m}-\inf_{B_\epsilon(x_0)}u_{n}$, are not
larger than: $\sup_{B_\epsilon(x_0)}(u_m-u_n)$. 

It follows that $M<\alpha M+(\alpha+3)\delta$, which is a
contradiction with $M>0$ for $\delta$ sufficiently small, in view of
$\alpha<1$. Therefore, the convergence of $\{u_n\}$ to $u$ is uniform
and we have: $u=\lim_{n\to\infty} u_{n+1} = \lim_{n\to\infty}
Tu_n=T(\lim_{n\to\infty} u_n)=Tu$, which concludes the proof of
existence. 

\smallskip

{\bf 3.} We now prove uniqueness of solutions to \eqref{minmax}. Assume, by
contradiction, that $u$ and $\bar{u}$ are distinct solutions and denote: 
$$M=\sup_{\Omega} ~(u-\bar{u})>0.$$
Let $\{x_n\}_{n=1}^\infty$ be a sequence of points in $X$ such that
$\lim_{n\to\infty}(u-\bar{u})(x_n)=M$. Observe that $u(x_n)>\Psi_1(x_0)$ and
$\bar{u}(x_n)<\Psi_2(x_n)$ for large $n$. Without loss of generality,
$\{x_n\}$ converges to some $x_0\in\bar\Omega$. Therefore, as in \eqref{ineqCont1}, we get: 
\begin{equation*}
\begin{split}
(u-\bar{u})(x_n)&\leq
\frac{\alpha}{2}\Big(\sup_{B_\epsilon(x_n)}u-\sup_{B_\epsilon(x_n)}\bar{u}\Big)
+ \frac{\alpha}{2}\Big(\inf_{B_\epsilon(x_n)}u-\inf_{B_\epsilon(x_n)}\bar{u}\Big)
+ \beta\fint_{B_\epsilon(x_n)}(u-\bar{u})\\ 
&\leq\alpha \sup_{B_\epsilon(x_n)}\left(u-\bar{u}\right) + \beta\fint_{B_\epsilon(x_n)}(u-\bar{u})
\leq\alpha M + \beta\fint_{B_\epsilon(x_n)}(u-\bar{u}).
\end{split}
\end{equation*}
Passing to the limit with $n\to \infty$ yields: $M\leq\alpha M + \beta\fint_{B_\epsilon(x_0)}(u-\bar{u})$,
and thus: $M\leq\fint_{B_\epsilon(x_0)} (u-\bar u)$ in view of
$\beta>0$. The set $G=\left\{x\in X;~(u-\bar{u})(x)=M \right\}$
must therefore be dense in $B_\epsilon(x_0)$. By the
same argument we conclude that for  all $x\in G\cap\Omega$, the set
$B_{\epsilon}(x)\setminus G$ has measure $0$.  
After finitely many steps of such reasoning, we obtain a contradiction
with $u=\bar u =F$ in $\Gamma$.
\endproof

\bigskip

Finally, we note the following comparison principle, in view of the iteration
procedure (\ref{2.2})  for the unique solution to (\ref{minmax}):

\begin{lemma}\label{comparison}
Let $\alpha$ and $\beta$ be as in Theorem \ref{epsilonp}. Let
$u_\epsilon$ be the unique solution to (\ref{minmax}) with the data $F,
\Psi_1, \Psi_2$, while $\tilde u_\epsilon $ be the unique solution to
(\ref{minmax}) with the data $\tilde F, \tilde\Psi_1,
\tilde\Psi_2$. Assume that:
$$ F\leq \tilde F \quad \mbox{ in } \Gamma \qquad \mbox{ and }\qquad 
\Psi_1\leq \tilde \Psi_1, \quad \Psi_2\leq \tilde \Psi_2 \quad \mbox{ in } X.$$
Then: $u_\epsilon \leq \tilde u_\epsilon $ in $X$.
\end{lemma}

\section{The main convergence result: a proof of Theorem \ref{main}}\label{sec3}

\begin{lemma}\label{aa}
The Borel functions $u_\epsilon$ satisfy:
\begin{itemize}
\item[{(i)}] (Uniform boundedness):
$$\exists C>0\quad \forall\epsilon>0\qquad \|u_\epsilon\|_{L^\infty(\bar\Omega)} < C$$
\item[{(ii)}] (Uniformly vanishing discontinuities):
\begin{equation}\label{asc}
\forall \eta>0\quad \exists r_0, \epsilon_0> 0 \qquad
\forall\epsilon<\epsilon_0 \quad \forall x_0, y_0\in\bar{\Omega} 
\qquad |x_0-y_0|<r_0\Rightarrow |u_\epsilon(x_0) - u_\epsilon(y_0)| < \eta.
\end{equation}
\end{itemize}
\end{lemma}

\begin{proof}
{\bf 1.} Since for every $\epsilon>0$ we have $\Psi_1\leq
u_\epsilon\leq \Psi_2$ in $\bar\Omega$, it is clear that (i)
holds. Condition (ii) will be proved
by invoking the same result, already established for the approximate
solutions of the single obstacle problem, studied in
\cite{LM}. In fact, proving (\ref{asc}) was the main technical
ingredient in \cite{MPR, LM}, necessitating a careful estimate of the variation
of $u_\epsilon$ close to the boundary $\partial\Omega$. 
It involved designing specific strategies in the game-theoretical
interpretation of the discrete min-max equation (see section \ref{sec4}), comparison with the
fundamental solution under mixed boundary conditions and estimating the
exit time. 

Here, we bypass this direct analysis through the following construction.
Fix $\eta>0$. Let $\bar u_\epsilon$ be the unique solution to
(\ref{minmax}) with the same data $F$ and $\Psi_1$, but with the
new upper obstacle $\bar\Psi_2 \equiv \sup_X \Psi_2$. Since $\bar
u_\epsilon \leq \bar\Psi_2$ and $\bar\Psi_2$ is a constant, it follows that:
\begin{equation}\label{obst1}
\bar u_\epsilon(x) = 
\begin{cases}
\max\Big\{\Psi_1(x), \displaystyle{\frac{\alpha}{2}\sup_{B_\epsilon(x)}
\bar u_\epsilon + \frac{\alpha}{2}\inf_{B_\epsilon(x)} \bar u_\epsilon +
\beta\fint_{B_\epsilon(x)} \bar u_\epsilon}\Big\} &   \mbox{for }  x\in\Omega, \\
F(x) & \mbox{for }  x\in\Gamma,
\end{cases}
\end{equation}
that is $\bar u_\epsilon$ is the unique solution of the approximation
(\ref{obst1}) to the single obstacle problem with data $F$ and
$\Psi_1$. By \cite[Corollary 4.5]{LM} we thus get:
\begin{equation}\label{asc1}
\exists r_0, \epsilon_0> 0 \qquad
\forall\epsilon<\epsilon_0 \quad \forall x_0, y_0\in\bar{\Omega} 
\qquad |x_0-y_0|<r_0\Rightarrow |\bar u_\epsilon(x_0) - \bar u_\epsilon(y_0)| < \eta.
\end{equation}

Likewise, let ${\underline u}_\epsilon$ be the unique solution to
(\ref{minmax}) with the same $F$ and $\Psi_2$ but with a new lower
obstacle ${\underline\Psi}_1 \equiv \inf_X\Psi_1$. Again, since
${\underline u}_\epsilon\geq {\underline \Psi}_1$, we trivially obtain:
\begin{equation}\label{obst2}
{\underline u}_\epsilon(x) = 
\begin{cases}
\min\Big\{\Psi_2(x), \displaystyle{\frac{\alpha}{2}\sup_{B_\epsilon(x)}
{\underline u}_\epsilon + \frac{\alpha}{2}\inf_{B_\epsilon(x)} {\underline u}_\epsilon +
\beta\fint_{B_\epsilon(x)} {\underline u}_\epsilon}\Big\} &   \mbox{for }  x\in\Omega, \\
F(x) & \mbox{for }  x\in\Gamma.
\end{cases}
\end{equation}
It follows that $(-{\underline u}_\epsilon)$ is the unique solution to
the approximation (\ref{obst2}) of the single obstacle problem with
boundary data $(-F)$ and the lower obstacle $(-\Phi_2)$.
Again, by \cite{LM} and possibly decreasing the values $r_0,
\epsilon_0>0$ in (\ref{asc1}), we obtain:
\begin{equation}\label{asc2}
\forall x_0, y_0\in\bar{\Omega} 
\qquad |x_0-y_0|<r_0\Rightarrow |{\underline u}_\epsilon(x_0) - {\underline u}_\epsilon(y_0)| < \eta.
\end{equation}

\smallskip

Note now that by Lemma \ref{comparison} there must be:
$${\underline u}_\epsilon \leq u_\epsilon \leq \bar u_\epsilon \qquad
\mbox{ in } \bar\Omega.$$
Consequently, for any $x_0\in\bar\Omega$ and $y_0\in\partial\Omega$
such that $|x_0 - y_0|<r_0$, we get:
\begin{equation*}
\begin{split}
& u_\epsilon(x_0) - u_\epsilon(y_0) \leq \bar u_\epsilon(x_0) - F(y_0) =  \bar u_\epsilon(x_0) - \bar u_\epsilon(y_0) 
<\eta,\\
& u_\epsilon(x_0) - u_\epsilon(y_0) \geq {\underline u}_\epsilon(x_0) -
F(y_0) =  {\underline u}_\epsilon(x_0) - {\underline u}_\epsilon(y_0)  > -\eta,
\end{split}
\end{equation*}
which yields: $|u_\epsilon(x_0) - u_\epsilon(y_0) |<\eta$.

\medskip

{\bf 2.} We now justify the validity of (\ref{asc}) for arbitrary
$x_0, y_0\in\Omega$ by transfering the boundary estimates to the interior of
the domain $\Omega$. This is done as in the proof of \cite[Corollary 4.5]{LM}.
Fix $\eta>0$. In view of the first part of the proof, as well as the
Lipschitzeanity of $F, \Psi_1$ and $\Psi_2$, we may find $r_0, \epsilon_0>0$ such that:
\begin{equation}\label{4.6}
\begin{split}
& \forall \epsilon < \epsilon_0\quad \forall x_0\in\bar\Omega, ~
y_0\in\partial\Omega\qquad 
|x_0-y_0|<r_0\Rightarrow |u_\epsilon(x_0) - u_\epsilon(y_0)| < \frac{\eta}{4}\\
& \forall x_0, y_0\in X \qquad |x_0-y_0|<r_0\Rightarrow |\Psi_1(x_0)-\Psi_1(y_0)|, ~
|\Psi_2(x_0)-\Psi_2(y_0)| <\frac{\eta}{4}\\
& \forall x_0, y_0\in X \qquad  |x_0-y_0|<r_0\Rightarrow ~ |F(x_0) - F(y_0)| < \frac{\eta}{4}.
\end{split}
\end{equation}

Call: $~\tilde\Gamma = \{x\in\bar\Omega; ~
\mbox{dist}(x, \partial\Omega)\leq r_0/2\}$ and note that by (\ref{4.6}):
\begin{equation}\label{mam}
\forall \epsilon < \epsilon_0\quad \forall x_0, y_0\in\tilde \Gamma\qquad 
|x_0-y_0|<c{r_0}\Rightarrow |u_\epsilon(x_0) - u_\epsilon(y_0)| < \frac{3}{4}\eta,
\end{equation}
for an appropriately small constant $c\in (0,1)$. Fix arbitrary $x_0, y_0\in\Omega$ with $|x_0 -
y_0|<c r_0$, and for any $\epsilon<\epsilon_0$ define the bounded Borel
functions $\tilde F:\tilde\Gamma\to\R$ and $\tilde\Psi_1, \tilde\Psi_2:\R^N\to \R$ by:
\begin{equation*}
\begin{split}
& \tilde{F}(z)=u_\epsilon(z-(x_0-y_0)) + \frac{3}{4}\eta\\
& \tilde{\Psi}_1(z)=\Psi_1(z-(x_0-y_0)) + \frac{3}{4}\eta, \qquad
\tilde{\Psi}_2(z)=\Psi_2(z-(x_0-y_0)) + \frac{3}{4}\eta.
\end{split}
\end{equation*}
Let $\tilde{u}_\epsilon$ be the unique solution to the min-max principle as in Theorem \ref{epsilonp}:
\begin{equation*}
\tilde u_\epsilon(x) = 
\begin{cases}
\max\Bigg\{\tilde\Psi_1(x),\min\Big\{\displaystyle{\tilde\Psi_2(x),\frac{\alpha}{2}\sup_{B_\epsilon(x)}
\tilde u_\epsilon + \frac{\alpha}{2}\inf_{B_\epsilon(x)} \tilde u_\epsilon +
\beta\fint_{B_\epsilon(x)}\tilde u_\epsilon}\Big\}\Bigg\} &
\mbox{for }  x\in\Omega\setminus \tilde\Gamma, \\
\tilde F(x) & \mbox{for }  x\in\tilde \Gamma.
\end{cases}
\end{equation*}
By uniqueness of such solution, there must be:
$$\tilde u_\epsilon(z) = u_\epsilon(z-(x_0-y_0)) + \frac{3}{4}\eta
\qquad \mbox{ in } \Omega.$$
On the other hand, since in view of (\ref{4.6}) and (\ref{mam}) there
is: $\tilde F\geq u_\epsilon$ in $\tilde\Gamma$ and $\tilde\Psi_1\geq \Psi_1$, 
$\tilde \Psi_2\geq \Psi_2$ in $\mathbb{R}^N$, Lemma \ref{comparison}
implies that: $\tilde u_\epsilon \geq u_\epsilon$ in  $\bar\Omega$. Thus: 
$$\forall\epsilon<\epsilon_0\qquad u_\epsilon(x_0) - u_\epsilon(y_0)
\leq \tilde u_\epsilon(x_0) - u_\epsilon(y_0) = u_\epsilon (y_0) +
\frac{3}{4}\eta - u_\epsilon(y_0) < \eta.$$
Exchanging $x_0$ with $y_0$, the same argument yields: $u_\epsilon(y_0)
- u_\epsilon(x_0) <\eta$, achieving the Lemma.
\end{proof}

\bigskip

We are now ready to give:

\bigskip

\noindent{\bf Proof of Theorem \ref{main}.}

By Lemma \ref{aa} and in virtue of the Ascoli-Arzel\`a type of result in
\cite[Lemma 4.2]{MPR}, it follows that $\{u_\epsilon\}$ has a subsequence
converging uniformly in $\bar\Omega$ to a continuous function
$u:\bar\Omega\to\R$. We now show that $u$ is a viscosity solution of (\ref{obst}). By uniqueness
of such solutions that will be shown in Theorem \ref{thuni}, 
we will conclude that the whole sequence $\{u_\epsilon\}$ converges to the same limit $u$. 

In order to prove (ii), assume that $\Psi_1(x_0) < u(x_0)$. By
continuity of $\Psi_1$ and $u$, we obtain that also: $\Psi_1 <
u_\epsilon$ in some $B_\delta(x_0)\subset\Omega$ and for all small
$\epsilon<\epsilon_0$. By (\ref{minmax}) we then get:
\begin{equation}\label{obst3}
\forall x\in B_\delta(x_0)\qquad -u_\epsilon(x) = \max\Big\{
-\Psi_2(x), \displaystyle{\frac{\alpha}{2}\sup_{B_\epsilon(x)}
(-u_\epsilon) + \frac{\alpha}{2}\inf_{B_\epsilon(x)} (-u_\epsilon) +
\beta\fint_{B_\epsilon(x)}(-u_\epsilon)}\Big\}.
\end{equation}
Applying the proof of \cite[Theorem 1.2]{LM}, we directly conclude (ii), since
$\{u_\epsilon\}$ satisfies the discrete approximation
(\ref{obst3}) of the single lower obstacle $(-\Psi_2)$ problem in a neighbourhood of $x_0$.

Likewise, to justify (iii) let $u(x_0) < \Psi_2(x_0)$. We have:
$$\exists \delta, \epsilon_0>0\quad \exists C\qquad \forall
\epsilon<\epsilon_0\quad \forall x\in B_\delta(x_0)\qquad
u_\epsilon(x) < C < \Psi_2(x) $$
for some constant $C$. Consequently, (\ref{minmax}) implies:
$$\forall x\in B_\delta(x_0)\qquad u_\epsilon(x) = \max\Big\{
\Psi_1(x), \displaystyle{\frac{\alpha}{2}\sup_{B_\epsilon(x)}
u_\epsilon + \frac{\alpha}{2}\inf_{B_\epsilon(x)} u_\epsilon +
\beta\fint_{B_\epsilon(x)} u_\epsilon}\Big\}$$
and so $\Delta_p\phi(x_0)\leq 0$ for any appropriate test
function $\phi$ supporting $u$ from below at $x_0$. Indeed, one may
apply the local argument in \cite[Theorem 1.2]{LM} to the obstacle
problem with the single lower obstacle $\Psi_1$.
\endproof

\section{Uniqueness of viscosity solutions to the double-obstacle
  problem (\ref{obst})} \label{secuni}

We start by recalling the following result, due to Farnana in \cite{farn}:

\begin{theorem}\label{thf}
Let $p,\Psi_1, \Psi_2, F$ be as in Theorem \ref{main}. Define the set:
$$\mathcal{K}_{\Psi_1,\Psi_2,F}(\Omega) = \Big\{u\in W^{1,p}(\Omega);
  ~ u=F \mbox{ on } \partial\Omega ~\mbox{ and } ~\Psi_1\leq u\leq
  \Psi_2 \mbox{ in } \Omega\Big\}.$$
\begin{itemize}
\item[(i)] There exists a unique $u\in \mathcal{K}_{\Psi_1,\Psi_2,F}(\Omega)$ such that:
\begin{equation}\label{var}
\int_\Omega |\nabla u|^p\leq \int_\Omega |\nabla v|^2 \qquad \forall v\in \mathcal{K}_{\Psi_1,\Psi_2,F}(\Omega).
\end{equation}
\item[(ii)] The unique minimizer $u$ in (\ref{var}) is continuous: $u\in\mathcal{C}(\bar \Omega)$.
\item[(iii)] Let $\bar u$ be the unique minimizer, as above, for the
  data $\bar\Psi_1$, $\bar\Psi_2$ and $\bar F$. If  $\bar\Psi_1\leq\Psi_1$, $\bar\Psi_2\leq \Psi_2$ 
and $\bar F\leq F$, then $\bar u\leq u$ in $\bar\Omega$.
\end{itemize}
\end{theorem}

We remark that existence and uniqueness of the variational solution in
(\ref{var}) is an easy direct consequence of the strict convexity of
the functional $\int |\nabla u|^p$. The regularity
and comparison principle statements in (ii) and (iii) were proved in
\cite{farn} in the generalized setting of the dounble obstacle problem on metric spaces.

\medskip

A standard calculation easily shows that the unique variational
solution to the double-obstacle problem as in Theorem \ref{thf} (i), must be
a viscosity solution in the sence of Definition \ref{viscsol}. Therefore,
in view of uniqueness, proved below, the two notions actually coincide.
Here is the main result of this section:

\begin{theorem}\label{thuni}
Let $p,\Psi_1, \Psi_2, F$ be as in Theorem \ref{main}.  Let $u$ and
$\bar u$ be two viscosity solutions to (\ref{obst}) as in Definition \ref{viscsol}. Then $u=\bar u$.
\end{theorem}

\begin{proof}
{\bf 1.} Let $\mathcal{U}$ be any open, Lipschitz set such that:
$$\mathcal{U}\subset\subset \{x\in\Omega;~ \Psi_1(x)\neq
\Psi_2(x)\}.$$ 
We will show that  $u$ as in the statement of the Theorem is the variational solution to
the double-obstacle problem on $\mathcal{U}$, in the sence of
(\ref{var}) in the set $\mathcal{K}_{\Psi_1, \Psi_2,  u_{\mid \partial\mathcal{U}}}(\mathcal{U})$.

Firstly, note that on the open set $\mathcal{U}_2=\{x\in\mathcal{U}; ~ u(x)<\Psi_2(x)\}$,
the continuous function $u$ is a viscosity $p$-supersolution to
(\ref{plap}). Thus, by the celebrated result in \cite{JLM}, $u$ is
$p$-superharmonic in $\mathcal{U}_2$ and consequently (see \cite{L})
$u \in W^{1,p}_{loc}(\mathcal{U}_2)$. In the same manner, $u$ is
a viscosity $p$-subsolution on $\mathcal{U}_1=\{x\in\mathcal{U}; ~ u(x)>\Psi_1(x)\}$,
hence it is $p$-subharmonic in $\mathcal{U}_1$ and $u \in W^{1,p}_{loc}(\mathcal{U}_1)$. 
Observing that $\mathcal{U}=\mathcal{U}_1\cup\mathcal{U}_2$ we obtain that $u
\in W^{1,p}_{loc}(\mathcal{U})$. Repeating the same argument on
$\tilde{\mathcal{U}}\supset\supset \mathcal{U}$ we conclude that
actually $u \in W^{1,p}(\mathcal{U})$. 

Recall that for a continuous function with regularity $W^{1,p}$, the
notions of $p$-superharmonic ($p$-subharmonic) and weak supersolution
(respectively weak subsolution) agree \cite{L}. We thus get:
\begin{equation}\label{lu1}
\int_{\mathcal{U}_1}|\nabla u|^p\leq \int_{\mathcal{U}_1}|\nabla
(u+\phi)|^p \qquad \forall \phi\in\mathcal{C}_0^\infty(\mathcal{U}_1,\mathbb{R}_+),
\end{equation}
\begin{equation}\label{lu2}
\int_{\mathcal{U}_2}|\nabla u|^p\leq \int_{\mathcal{U}_2}|\nabla
(u+\phi)|^p \qquad \forall \phi\in\mathcal{C}_0^\infty(\mathcal{U}_2,\mathbb{R}_-).
\end{equation}

Let now $\phi\in\mathcal{C}_0^\infty(\mathcal{U},\mathbb{R})$ be such
that $\Psi_1\leq u+\phi\leq\Psi_2$. We write: $\phi = \phi^+ +\phi^-$
as the difference of the positive and negative parts of $\phi$. Denote:
\begin{equation*}
D^+ = \{x\in\mathcal{U}; ~ \phi(x)>0\}\subset\mathcal{U}_1\quad\mbox{
  and } \quad D^- = \{x\in\mathcal{U}; ~ \phi(x)<0\}\subset\mathcal{U}_2.
\end{equation*}
Then we have:
\begin{equation}\label{lu3}
\begin{split}
& \int_{\mathcal{U}}|\nabla u +\nabla \phi|^p  = \int_{D^+}|\nabla u
+\nabla \phi|^p + \int_{D^-}|\nabla u +\nabla \phi|^p +
\int_{\{\phi=0\}}|\nabla u|^p \\ & =
\int_{\mathcal{U}_1}|\nabla u +\nabla (\phi^+)|^p - \int_{\mathcal{U}_1\setminus D^+}|\nabla u|^p 
+ \int_{\mathcal{U}_2}|\nabla u +\nabla (\phi^-)|^p - \int_{\mathcal{U}_2\setminus D^-}|\nabla u|^p 
+ \int_{\{\phi=0\}}|\nabla u|^p \\ & \geq \int_{\mathcal{U}_1}|\nabla u|^p - \int_{\mathcal{U}_1\setminus D^+}|\nabla u|^p  
+ \int_{\mathcal{U}_2}|\nabla u|^p - \int_{\mathcal{U}_2\setminus D^-}|\nabla u|^p 
+ \int_{\{\phi=0\}}|\nabla u|^p = \int_{\mathcal{U}}|\nabla u|^p, 
\end{split}
\end{equation}
where the inequality above  follows from (\ref{lu1}) and (\ref{lu2})
that are still valid with the test functions $\phi^+\in
W_0^{1,2}(\mbox{supp } \phi,\mathbb{R}_+)$ and 
$\phi^-\in W_0^{1,2}(\mbox{supp } \phi,\mathbb{R}_-)$.

\medskip

{\bf 2.} Let now $u$ and $\bar u$ be two viscosity solutions to the
problem (\ref{obst}). Note that on the closed (and possibly very
irregular) set $A=\{x\in\bar\Omega;~ \Psi_1(x)=\Psi_2(x)\} \cup\partial\Omega$ 
we have $u=\bar u$. 

Fix $\epsilon>0$. By the uniform continuity of
$u$, $\bar u$ on $\Omega$, there exists $\delta>0$ such that:
\begin{equation}\label{lu4}
|u(x)-\bar u(x)|\leq \epsilon \qquad\forall x\in\mathcal{O}_\delta(A)
:= \big(A+B(0,\delta)\big)\cap\bar\Omega.
\end{equation}
Consider an arbitrary open, Lipschitz set  $\mathcal{U}$  satisfying:
$$\Omega\setminus \mathcal{O}_\delta(A)\subset\subset \mathcal{U}
\subset\subset \Omega\setminus A.$$
By the argument in Step 1, $u$ is the variational solution as in (\ref{var})
in the set $\mathcal{K}_{\Psi_1,\Psi_2, u_{\mid \partial\mathcal{U}}}(\mathcal{U})$, and $\bar u+\epsilon$
is the variational solution in the set $\mathcal{K}_{\Psi_1,
  \Psi_2+\epsilon, \bar{u}_{\mid\partial\mathcal{U}}+\epsilon}(\mathcal{U})$.
Since $u<\bar u+\epsilon$ on $\partial\mathcal{U}$ in view of
(\ref{lu4}), the comparison principle in Theorem \ref{thf} (iii) implies now that
$u\leq\bar u+\epsilon $ in $\bar{\mathcal{U}}$. 

Reversing thesame  argument and taking into account (\ref{lu4}), we arrive at:
$$|u(x) - \bar u(x)|\leq \epsilon \qquad \forall x\in\bar\Omega.$$
We conclude that $u=\bar u$ in $\bar\Omega$ passing to the limit
$\epsilon\to 0$ in the above bound.
\end{proof}

\section{The tug-of-war game with double stopping times}\label{sec4}

Consider the following game, played by Player I and Player II on the
board given by the set $X$ and with the initial poition of the token
$x_0\in X$. At each turn of the game, a coin is flipped in order to
determine which player is in charge. The chosen player is allowed to move
the token to any point in an open ball of radius $\epsilon$ around the
current position $x_n$. He is also allowed to forfeit the move and stop the
game instead. If Player I stops the game then the payoff is
$\Psi_1(x_n)$, while is Player II stops, then the payoff is
$\Psi_2(x_n)$. If neither player decides to stop the game, it is
stopped when the token reaches the boundary $\Gamma$. In this case the
payoff is $F(x_n)$. The payoff is always awarded to Player I and
penalizes Player II (this is a zero-sum game), so that Player I will try to maximize and Player II to minimize it.

We now show that solutions $u_\epsilon$ of (\ref{minmax}) coincide
with the expected value of the above game, when both players play optimally. 
We begin by introducing the necessary probability framework.
	
\subsection{The measure spaces} 
Fix $x_0\in X$ and define:
$$X^{\infty,x_0}=\big\{\omega=(x_0,x_1,x_2...);~ x_n\in X \mbox{ for all } n\geq 1\big\},$$
to be the space of all infinite game runs, recording by $x_n\in X$ 
the position of the token at the $n$-th step of the game. 
For each $n\geq 1$, let $\mathcal{F}^{x_0}_n$ be the $\sigma$-algebra of subsets 
of $X^{\infty, x_0}$ generated by all sets consisting of game runs of length $n$:
\begin{equation}\label{conv}
A_1\times \ldots \times A_n := \big\{\omega \in \{x_0\}\times
A_1\times \ldots\times A_n\times X\times X\times \ldots \big\},
\end{equation}
where $A_1,\ldots A_n$ are Borel subsets of $X$. We then define
$\mathcal{F}^{x_0}$ as the $\sigma-$algebra of subsets of $X^{\infty,
  x_0}$ generated by $\bigcup_{n=1}^\infty \mathcal{F}^{x_0}_n$. 
Clearly, the increasing sequence $\{\mathcal{F}^{x_0}_n\}_{n\geq 1}$
is a filtration of $\mathcal{F}^{x_0}$, and the coordinate projections
$x_n:X^{\infty,x_0}\to X$ given by: $x_n(\omega) = x_n$ are
$\mathcal{F}^{x_0}_n$ measurable.
	
\subsection{The strategies} 	
For every $n\geq 0$, let $\sigma_{I}^n, \sigma_{II}^n:X^{n+1}\to X$ be Borel measurable
functions, indicating the position of the token if it is moved by
Player I or Player II, respectively, at the $n$-th step of the game
given the history $(x_0,\ldots x_n)$. We assume that:
$$\sigma_{I}^n (x_0, \ldots x_n),~\sigma_{II}^n(x_0,\ldots ,x_n)\in
B_\epsilon(x_n)\cap X$$ 
and we call the collections $\sigma_{I}=\{\sigma_{I}^n\}_{n\geq 0}$ and
$\sigma_{II}=\{\sigma_{II}^n\}_{n\geq 0}$ the strategies of Players I and II.

\subsection{The stopping times} 
Recall that a random variable $\tau:X^{\infty,x_0}\to \mathbb{N}\cup\{+\infty\}$ is a stopping
time with respect to the filtration $\{\mathcal{F}_n^{x_0}\}$ if
$\tau^{-1}(\{0,1,\ldots n\}) \in \mathcal{F}^{x_0}_n$ for all $n\geq
1$. We define:
$$A^\tau_n=\big\{(x_0,...,x_n);~\exists\omega=(x_0,\ldots,
x_n,x_{n+1},\ldots)\in X^{\infty, x_0}, ~ \tau(\omega)\leq n\big\}.$$ 
Let $\tau_I,\tau_{II}$ be two stopping times as above, chosen by
Players I and II. We assume that they both do not exceed the exit time
from $\Omega$, i.e.:
$$\forall\omega\in X^{\infty, x_0}\qquad \tau_I(\omega),
\tau_{II}(\omega) \leq \tau_0(\omega) = \min\{n\geq 0;~ x_n(\omega)\in\Gamma\},$$ 
with the convention that the minimum over the empty set
is $+\infty$. For every $n\geq 0$ we then define:
$$A_n^{\tau_I < \tau_{II}} = \bigcup_{k=1}^n \Big(A_k^{\tau_I}\setminus A_k^{\tau_{II}}\Big).$$

\subsection{The probability measures} 
Fix two parameters $\alpha, \beta\geq 1$ with $\alpha+\beta=1$. Given
strategies $\sigma_I, \sigma_{II}$ and a stopping time $\tau\leq
\tau_0$ as above, we  define a family of ``transition'' probability (Borel)
measures on $X$. Namely, for $n\geq 1$ and every finite history
$(x_0,\ldots, x_n)\in X^{n+1}$ we set:
\begin{equation}\label{trans}
\gamma_n[x_0,\ldots,x_n] =
\begin{cases}
{\displaystyle \frac{\alpha}{2}\delta_{\sigma_{I}^n(x_0,\ldots,x_n)} +
\frac{\alpha}{2}\delta_{\sigma_{II}^n(x_0,\ldots,x_n)} + 
\beta\frac{\mathcal{L}_N\lfloor B_\epsilon(x_n)}{|B_\epsilon|}} & \mbox{for } (x_0,\ldots,x_n)\not\in A^{\tau}_n,\\
\delta_{x_n} & \mbox{otherwise}.
\end{cases}
\end{equation}
Above, $\delta_{y}$ stands for the Dirac delta at a given point $y\in X$, 
while $\frac{\mathcal{L}_N\lfloor B_\epsilon(x_n)}{|B_\epsilon|}$
denotes the $N$-dimensional Lebesgue measure restricted to the ball
$B_\epsilon(x_n)$ and normalised by its volume.

Note that the family (\ref{trans}) is jointly measurable, in the sense
that  for every $n\geq 1$ and every fixed Borel set $A\subset  X$,
the function:
$$ X^{n+1}\ni (x_0,\ldots, x_n) \mapsto \gamma_n[x_0,\ldots, x_n](A)
\in\R$$
is Borel measurable. Thus, we the probability measure
$\mathbb{P}_{\sigma_I,\sigma_{II},\tau}^{n, x_0}$ on $(X^{\infty,
  x_0}, \mathcal{F}_n^{x_0})$ is well defined:
$$ \mathbb{P}_{\sigma_I,\sigma_{II}, \tau}^{n, x_0} (A_1\times\ldots\times A_n) 
= \int_{A_1}\ldots\int_{A_n} 1 ~\mbox{d}\gamma_{n-1}[x_0,\ldots,x_{n-1}] \ldots \mbox{d}\gamma_{0}[x_0],$$
for every $n$-tuple of Borel sets $A_1,\ldots, A_n\subset X$. 
The family $\{ \mathbb{P}_{\sigma_I,\sigma_{II}, \tau}^{n,  x_0}\}_{n\geq 1}$ 
is also consistent, so it generates (by Kolmogoroff's consistency theorem
\cite{Var}) the unique probability measure:
$$ \mathbb{P}_{\sigma_I,\sigma_{II}, \tau}^{x_0} = \lim_{n\to\infty}
\mathbb{P}_{\sigma_I,\sigma_{II}, \tau}^{n, x_0} $$
on $(X^{\infty, x_0}, \mathcal{F}^{x_0})$ such that, using the notation convention (\ref{conv}), we have:
$$\forall n\geq 1 \quad\forall A_1\times\ldots\times A_n\in
\mathcal{F}_n^{x_0}\qquad  \mathbb{P}_{\sigma_I,\sigma_{II},
  \tau}^{x_0}(A_1\times \ldots\times A_n) =  \mathbb{P}_{\sigma_I,\sigma_{II}, \tau}^{n, x_0}
(A_1\times \ldots\times A_n).$$

One can easily prove the following useful observation, which follows
by directly checking the definition of conditional expectation:

\begin{lemma}\label{uno}
Let $v:X\rightarrow \mathbb{R}$ be a bounded Borel function. For any $n\geq 1$,
the conditional expectation $\mathbb{E}^{x_0}_{\sigma_{I},
  \sigma_{II}, \tau}\{v\circ x_n\mid \mathcal{F}^{x_0}_{n-1}\}$
of the random variable $v\circ x_n$  is a  $\mathcal{F}^{x_0}_{n-1}$
measurable function on $X^{\infty, x_0}$ (and hence it depends only on
the initial $n$ positions in the history $\omega =
(x_0, x_1,\ldots )\in X^{\infty, x_0}$), given by:
$$\mathbb{E}^{x_0}_{\sigma_{I},
  \sigma_{II}, \tau}\{v\circ x_n\mid \mathcal{F}^{x_0}_{n-1}\} (x_0, \ldots,x_{n-1}) 
= \int_X v ~\mathrm{d}\gamma_{n-1}[x_0,\ldots, x_{n-1}].$$
\end{lemma}

\smallskip

We now invoke two useful results:

\begin{lemma}\label{due} \cite{LM}
In the above setting, assume that $\beta>0$. Then the game stops almost surely:
\begin{equation*}
\mathbb{P}^{x_0}_{\sigma_I,\sigma_{II}, \tau} \big(\{\tau<\infty\}\big) = 1.
\end{equation*}
\end{lemma}

\smallskip

\begin{lemma}\label{selection} \cite{LPS, LM}
Let $u:X\to\mathbb{R}$ be a bounded, Borel function. Fix $\delta,
\epsilon>0$. There exist Borel functions $\sigma_{sup},
\sigma_{inf}:\Omega\to X$ such that:
\begin{equation*}
\forall x\in\Omega \qquad \sigma_{sup}(x), \sigma_{inf}(x) \in B_\epsilon(x)
\end{equation*}
and:
\begin{equation*}
\forall x\in\Omega \qquad u(\sigma_{sup}(x)) \geq
\sup_{B_\epsilon(x)} u - \delta, \qquad 
u(\sigma_{inf}(x)) \leq \inf_{B_\epsilon(x)} u + \delta.
\end{equation*}
\end{lemma}

\subsection{The game value solves the dynamic programming principle (\ref{minmax})}

In the above setting, let $\beta > 0$ and let $\Psi_1, \Psi_2:\R^N\rightarrow\R$
and $F:{\Gamma}\rightarrow\R$ be bounded Borel functions such that $\Psi_1 \leq \Psi_2$ in $X$ and
$\Psi_1 \leq F \leq \Psi_2$ in $\Gamma$. Given two stopping times
$\tau_I, \tau_{II}\leq \tau_0$, define the sequence of Borel functions
$G_n^{\tau_I, \tau_{II}}: X^{n+1}\to\R$, for all $n\geq 1$ by:
\begin{equation}\label{ggg}
G_n^{\tau_I, \tau_{II}}(x_0, \ldots, x_n) = \begin{cases} F(x_n) & \mbox{ for } x_n\in \Gamma\\
\Psi_1(x_n) & \mbox{ for } x_n\in \Omega \mbox{ and } (x_0, \ldots x_n)\in A_n^{\tau_I< \tau_{II}}\\
\Psi_2(x_n) & \mbox{ otherwise. }
\end{cases}
\end{equation}
We will use the following notation:
$$\forall \omega\in X^{\infty, x_0}\qquad 
G_{\tau_I\wedge \tau_{II}}^{\tau_I, \tau_{II}}(\omega) = G_n^{\tau_I,
    \tau_{II}}(x_0, \ldots, x_n) \quad \mbox{ with } n = (\tau_I\wedge \tau_{II})(\omega)$$
for defining the two value functions:
\begin{equation}\label{val}
u_I(x_0) = \sup_{\sigma_I, \tau_I}\inf_{\sigma_{II}, \tau_{II}}
\mathbb{E}^{x_0}_{\sigma_I, \sigma_{II}, \tau_I\wedge
  \tau_{II}}\big[G^{\tau_I, \tau_{II}}_{\tau_I\wedge\tau_{II}}\big], \qquad 
u_{II}(x_0) = \inf_{\sigma_{II}, \tau_{II}} \sup_{\sigma_I, \tau_I}
\mathbb{E}^{x_0}_{\sigma_I, \sigma_{II}, \tau_I\wedge
  \tau_{II}}\big[G^{\tau_I, \tau_{II}}_{\tau_I\wedge\tau_{II}}\big]
\end{equation}
Note that in view of Lemma \ref{due}, the expectations in (\ref{val}) are well defined.

\medskip

The following is the main result of this section:

\begin{theorem}\label{ExpVal}
Let $\alpha, \beta, F, \Psi_1, \Psi_2$ be as in Theorem \ref{epsilonp}. Then we have:
$$u_{I}= u_{II} = u_\epsilon \qquad \mbox{ in } \Omega$$
where $u_\epsilon$ is the unique solution to (\ref{minmax}).
\end{theorem}
\begin{proof} 
{\bf 1.} We begin by proving that:
\begin{equation}\label{3.8}
u_{II}\leq u \qquad \mbox{ in } \Omega.
\end{equation}
Fix $\eta>0$ and let $\sigma_I$ and $\tau_I$ be any strategy and any
admissible stopping time chosen by Player I. Applying the selection
Lemma \ref{selection}, choose a Markovian strategy 
$\bar\sigma_{II}$ such that $\bar\sigma_{II}^n(x_0,\ldots, x_n) = \bar\sigma_{II}^n(x_n)$ and:
\begin{equation}\label{3.9}
\forall n\geq 0 \quad \forall x_n\in X \qquad
u(\bar\sigma_{II}^n(x_n))\leq \inf_{B_\epsilon(x_n)} u + \frac{\eta}{2^{n+1}}.
\end{equation}
Choose also the stopping time:
$$\bar\tau_{II}(\omega) = \inf\big\{ n \geq 0; ~ u(x_n) = \Psi_2(x_n)
~\mbox{ or }~ x_n\in\Gamma\big\}.$$ 

We will show that the sequence of random variables $\big\{u\circ
x_n+\frac{\eta}{2^{n}}\big\}_{n\geq 0}$ is a supermartingale with respect
to the filtration $\{\mathcal{F}_n^{x_0}\}$.  
Using Lemma \ref{uno} and the condition (\ref{3.9}), we obtain:
\begin{equation}\label{3.4}
\begin{split}
\forall (x_0,\ldots,x_{n-1})\not\in &A_{n-1}^{\tau_I\wedge\bar\tau_{II}}  \qquad
\mathbb{E}_{\sigma_I, \bar\sigma_{II},  \tau_I\wedge\bar\tau_{II}}^{x_0} \Big\{u \circ
x_n + \frac{\eta}{2^{n}}\mid \mathcal{F}_{n-1}^{x_0}\Big\}(x_0,\ldots,x_{n-1})\\
& =\int_{X} u ~\mbox{d}\gamma_{n-1}[x_0,\ldots,x_{n-1}] + \frac{\eta}{2^{n}}\\
& =\frac{\alpha}{2}u(\sigma_I^{n-1}(x_0,\ldots,x_{n-1})) +
\frac{\alpha}{2} u(\sigma_{II}^{n-1}(x_{n-1})) +\beta\fint_{B_\epsilon(x_{n-1})} u + \frac{\eta}{2^{n}} \\
& \leq\frac{\alpha}{2}\sup_{B_\epsilon(x_{n-1})}u +
\frac{\alpha}{2}\inf_{B_\epsilon(x_{n-1})} u +
\beta\fint_{B_\epsilon(x_{n-1})} u + \frac{\eta}{2^{n}}(\frac{\alpha}{2} +1 )\\
& \leq u(x_{n-1})+\frac{\eta}{2^{n-1}} = \big(u\circ x_{n-1} +
\frac{\eta}{2^{n-1}}\big) (x_0,\ldots, x_{n-1}), 
\end{split}
\end{equation}
where the last inequality above follows because: 
$$u(x_{n-1}) \geq \min\Big\{\Psi_2(x_{n-1}), \frac{\alpha}{2}\sup_{B_\epsilon(x_{n-1})}u +
\frac{\alpha}{2}\inf_{B_\epsilon(x_{n-1})} u +
\beta\fint_{B_\epsilon(x_{n-1})} u \Big\}$$
by (\ref{minmax}) and then there must be $ u(x_{n-1}) < \Psi_2(x_{n-1})$ since $(x_0,\ldots,
x_{n-1})\not\in A_{n-1}^{\bar\tau_{II}}$.  On the other hand, when
$(x_0,\ldots, x_{n-1})\in A_{n-1}^{\tau_I\wedge\bar\tau_{II}}$ then we directly get:
$$\mathbb{E}_{\sigma_I, \bar\sigma_{II},  \tau_I\wedge\bar\tau_{II}}^{x_0} \Big\{u \circ
x_n + \frac{\eta}{2^{n}}\mid
\mathcal{F}_{n-1}^{x_0}\Big\}(x_0,\ldots,x_{n-1}) = u(x_{n-1})+\frac{\eta}{2^{n}}. $$

By Doob's optional stopping theorem \cite{Var} applied to the uniformly
bounded random variables $\displaystyle{\Big\{u\circ x_{\tau_I\wedge\bar\tau_{II}\wedge
n} + \frac{\eta}{2^{\tau_I\wedge\bar\tau_{II}\wedge n}}\Big\}_{n\geq  0}}$, we obtain:
$$\mathbb{E}_{\sigma_I, \bar\sigma_{II},  \tau_I\wedge\bar\tau_{II}}\Big[u \circ x_{\tau_I\wedge\bar\tau_{II}}
+\frac{\eta}{2^{\tau_I\wedge\bar\tau_{II}}}\Big] \leq 
\mathbb{E}_{\sigma_I, \bar\sigma_{II},  \tau_I\wedge\bar\tau_{II}}\Big[u \circ x_0
+\frac{\eta}{2^0}\Big] = u(x_0) +\eta.$$
Consequently:
\begin{equation}\label{mart}
\begin{split}
u_{II}(x_0) & \leq \sup_{\sigma_I, \tau_I}\mathbb{E}_{\sigma_I,
  \bar\sigma_{II},  \tau_I\wedge\bar\tau_{II}}\left[G_{\tau_I\wedge\bar\tau_{II}}^{\tau_I,
  \bar\tau_{II}} + \frac{\eta}{2^{\tau_I\wedge\bar\tau_{II}}}\right]
\\ & \leq  \sup_{\sigma_I, \tau_I}\mathbb{E}_{\sigma_I,
  \bar\sigma_{II},  \tau_I\wedge\bar\tau_{II}}\left[u\circ x_{\tau_I\wedge\bar\tau_{II}}^{\tau_I,
  \bar\tau_{II}} + \frac{\eta}{2^{\tau_I\wedge\bar\tau_{II}}}\right]\leq u(x_0)+\eta,
\end{split}
\end{equation}
because for a given $\omega\in X^{\infty, x_0}$ such that
$n=(\tau_I\wedge \bar\tau_{II})(\omega)<+\infty$ there holds:
$$G_{\tau_I\wedge\bar\tau_{II}}^{\tau_I,  \bar\tau_{II}}(\omega) = G_n^{\tau_I,
  \bar\tau_{II}}(x_0, \ldots, x_n)\leq u(x_n).$$
The above inequality may be checked directly from the definition
(\ref{ggg}). For example, when $G_{\tau_I\wedge\bar\tau_{II}}^{\tau_I,
  \bar\tau_{II}}(\omega) = \Psi_2(x_n)$ then there must be
$x_n\in\Omega$ and $n=\bar\tau_{II}(\omega)\leq\tau_{I}(\omega)$, so
$u(x_n) = \Psi_2(x_n)$. This completes the proof of (\ref{3.8})
because $\eta>0$ was arbitrarily small.
	
\smallskip

{\bf 2.} Using the same reasoning as above, we now prove the second inequality: 
\begin{equation}\label{3.88}
u\leq u_{I} \qquad \mbox{ in } \Omega.
\end{equation}
Fix $\eta>0$ and let $\sigma_{II}$ and $\tau_{II}$ be any strategy and
an admissible stopping time for Player II. By Lemma \ref{selection}, we
choose a strategy $\bar\sigma_{I}$ so that $\bar\sigma_{I}^n(x_0,\ldots,x_n) = \bar\sigma_{I}^n(x_n)$
and:
$$\forall n\geq 0\quad \forall x_n\in X\qquad u(\bar\sigma_{I}^n(x_n))\geq \sup_{B_{\epsilon(x_n)}} u
-\frac{\eta}{2^{n+1}}.$$
We define the stopping time:
$$\bar\tau_{I}(\omega) = \inf\big\{ n \geq 0; ~ u(x_n) = \Psi_1(x_n)
~\mbox{ or }~ x_n\in\Gamma\big\}.$$ 
The sequence of random variables $\big\{u\circ x_n-\frac{\eta}{2^{n}}\big\}_{n\geq 0}$ is a submartingale 
with respect to the filtration $\{\mathcal{F}_n^{x_0}\}$. For the
proof, we reason as in (\ref{3.4}) and noting that for $(x_0,\ldots,
x_{n-1})\not\in A_{n-1}^{\bar\tau_I\wedge\tau_{II}}$ we have:
$u(x_{n-1}) > \Psi_1(x_{n-1})$, so by (\ref{minmax}) there must be:
$$u(x_{n-1}) \leq \frac{\alpha}{2}\sup_{B_\epsilon(x_{n-1})}u +
\frac{\alpha}{2}\inf_{B_\epsilon(x_{n-1})} u + \beta\fint_{B_\epsilon(x_{n-1})} u.$$
Further, using the same arguments as in (\ref{mart}), we obtain:
$$u_{I}(x_0) \geq u(x_0) - \eta,$$
where we used that $G_{\bar\tau_I\wedge\tau_{II}}^{\bar\tau_I, \tau_{II}}(\omega) \geq u(x_n)$
with $n= (\bar\tau_I\wedge\tau_{II})(\omega)$, for
$\mathbb{P}^\infty_{\bar\sigma_I, \sigma_{II},  \bar\tau_I\wedge\tau_{II}}$-almost every
$\omega \in X^{\infty, x_0}$.
Since $\eta>0$ was arbitrary, we indeed conclude (\ref{3.88}).
\end{proof}

\section{Numerical approximations of solutions to  (\ref{obst})}\label{secnum}

The approximation construction utilized in Theorem \ref{main} lends
itself very well to numerical use. Below, we set up a discretization of the
operator $T$ in (\ref{2.2}) and use it for approximating the solutions to (\ref{obst}).
	
\medskip

\noindent {\bf The algorithm.}
We consider the square domain $\Omega$ and the extended domain $X$:
$$\Omega=(-1,1)\times(-1,1)\subset \R^2, \qquad
X=\Omega\cup\Gamma=(-1.2,1.2)\times(-1.2,1.2),$$ 
where we set $\bar\epsilon_0=0.2$. A square mesh is
created in $X$ and we define the two initial iteration functions
$u_1^-$ and $u_1^+$ as equal to the lower and upper obstacle $\Psi_1$,
$\Psi_2$, respectively, on the mesh nodes $\Omega$ and both
equal to the boundary value $F$ on the mesh nodes in $\Gamma$. 

A discrete version $\bar T$ of the operator $T$ is defined as follows. 
Fix $\epsilon<\bar\epsilon_0$. Given a function $v$ on the nodes of
the mesh,  for every node $p\in\Omega$ we take all the nodes
$\{p_1,...,p_k\}$ in $X$ within $\epsilon$ distance of $p$ and evaluate: 
$$\bar{v}(p)=\frac{\alpha}{2}\max_{j=1\ldots k} v(p_j) +
\frac{\alpha}{2}\min_{j=1\ldots k} v(p_j) + \frac{\beta}{k}\sum\limits_{j=1}^{k}v(p_j).$$ 
The choice of $\epsilon$ affects the approximation and the speed of the
algorithm. We now set $\bar{T}v=v$ for nodes in $\Gamma$, while
for nodes $p\in\Omega$ we take:
\begin{equation} \label{DiscOper}
\bar{T}v(p)=\max\big\{u_1^-(p),\min\left\{u_1^+(p), \bar v(p)\right\}\big\}.
\end{equation}
The operator $\bar{T}$ is iterated to get two sequences of
functions: and increasing sequence $u_{n+1}^- = T(u_n^-)$ and a
decreasing sequence $u_{n+1}^+=T(u_n^+)$. We evaluate the maximum
difference between $u_n^-$ and $u_n^+$ and when it is less than the required accuracy, we break
the algorithm and return the values of $\frac{1}{2}(u_n^- + u_n^+)$ as the solution.  
		
\medskip

\noindent {\bf The resulting approximations of (\ref{obst}).} In
Figure \ref{fig1} we show the computed solutions for $p=2$ and
$p=100$, the boundary data $F=0$ and the obstacles:
\begin{equation}\label{try1}
\begin{split}
\Psi_1(x,y) &
=\max\Big\{1-33(x+0.5)^2-27(y+0.1)^2, ~ 0.5-40(x+0.3)^2-34(y+0.4)^2,
\\ & \qquad\qquad \qquad ~
0.5-36(x-0.6)^2-51(y-0.7)^2,~-2\Big\},\\ 
\Psi_2(x,y) &
=\min\Big\{33(x+0.6)^2+27(y-0.6)^2-1, ~33(x-0.6)^2+27(y+0.6)^2-1, ~2\Big\}.
\end{split}
\end{equation}

\begin{figure}[htp]
\centering
\hspace{-0.5cm}\includegraphics[width=9cm]{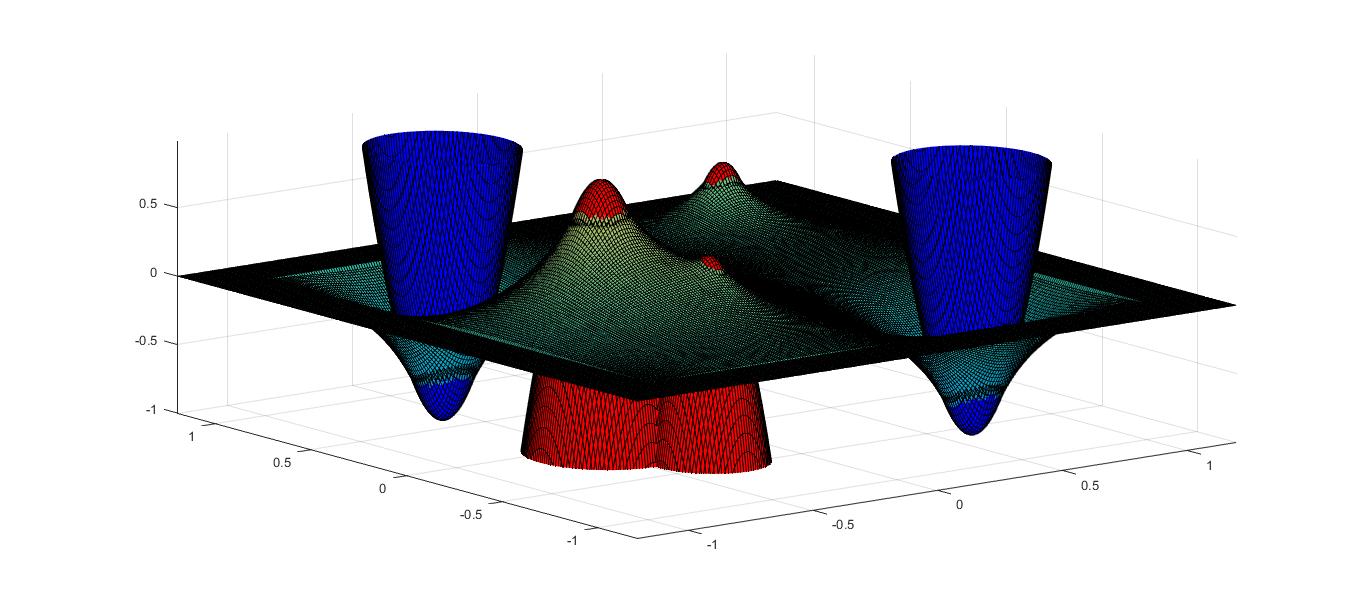}
\hspace{-1.5cm}
\includegraphics[width=9cm]{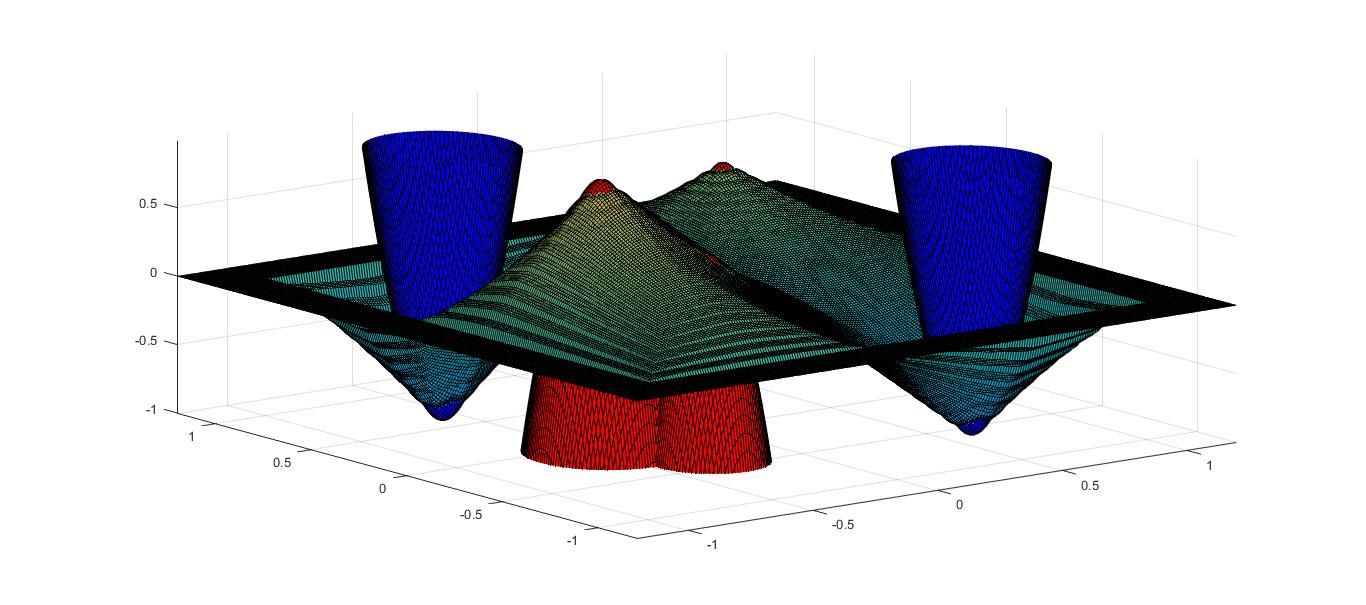}\\
\caption{Results of tests for $p=2$ and $p=100$ with data in (\ref{try1})} \label{fig1}
\end{figure}

\medskip

In Figure \ref{fig2} we show the computed solutions for $p=10$ and the
following sets of data: 

\medskip

(a) Smooth obstacles with parabolic boundary condition:
\begin{equation*}
\begin{split}
\Psi_1(x,y) & =\max\Big\{2-33(x+0.5)^2-27(y+0.1)^2, ~
1.5-40(x+0.3)^2-34(y+0.4)^2, \\ & \qquad\qquad\qquad ~ 2.5-36(x-0.6)^2-51(y-0.7)^2, ~-3\Big\},\\
\Psi_2(x,y) & =\min\Big\{33(x+0.6)^2+27(y-0.6)^2-3, ~
33(x-0.6)^2+27(y+0.6)^2-3, ~ 3\Big\},\\
F(x,y) & = 1-2y^2.
\end{split}
\end{equation*}

\medskip

(b) Lipschitz obstacles with zero boundary condition:
\begin{equation*}
\begin{split}
\Psi_1(x,y) & = \begin{cases}		
2-17|x-0.5|  & \quad \mbox{ for } y\in \left[-0.5,0.5\right]\\
2-17|x-0.5| - 17|y+0.5| &  \quad \mbox{ for } y\in (-1,-0.5)\\
2-17|x-0.5| - 17|y-0.5|  &\quad \mbox{ for } y\in (0.5,1)
\end{cases}\\
\Psi_2(x,y) & =-4+12|y+0.2|+15|x-0.7|\\
F(x,y) & =0.
\end{split}
\end{equation*}

\medskip

(c) Smooth obstacles with hyperbolic boundary condition, where
$\Psi_1$ and $\Psi_2$ are as in (a), and:
$$F(x,y) = 2-(x+y)^2.$$

\begin{figure}[htp]
\centering
\hspace{-0.5cm}\includegraphics[width=9cm]{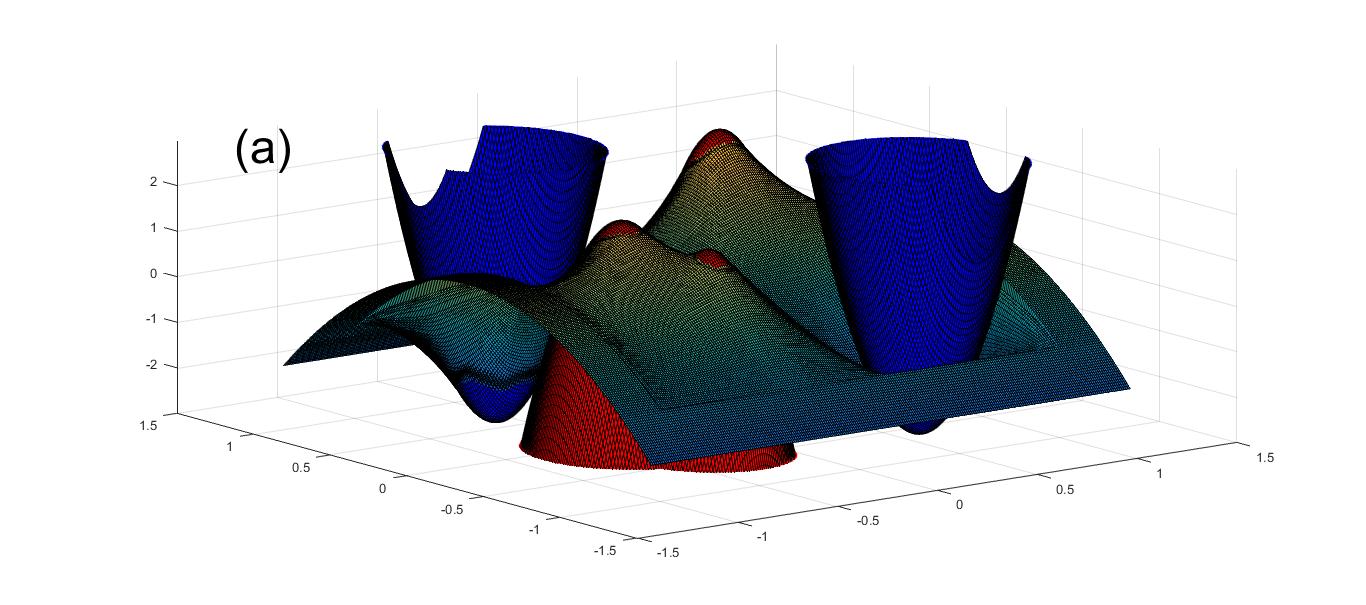}
\hspace{-1.5cm}
\includegraphics[width=9cm]{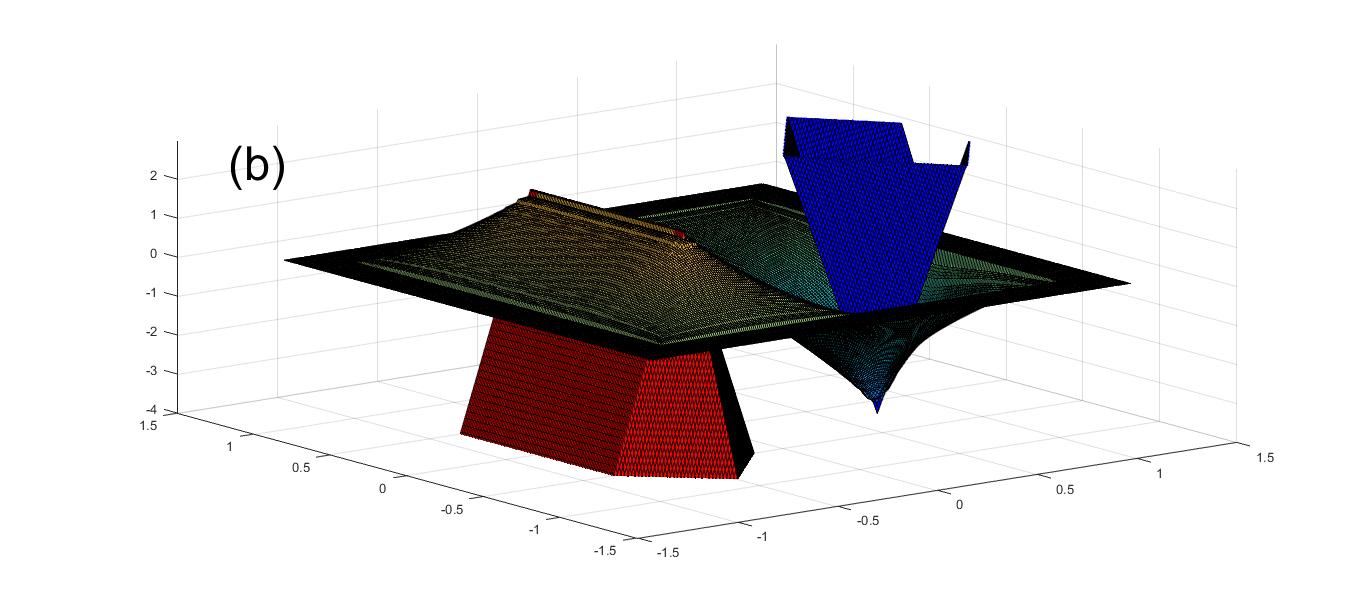}\\
\hspace{-0.5cm}\includegraphics[width=9cm]{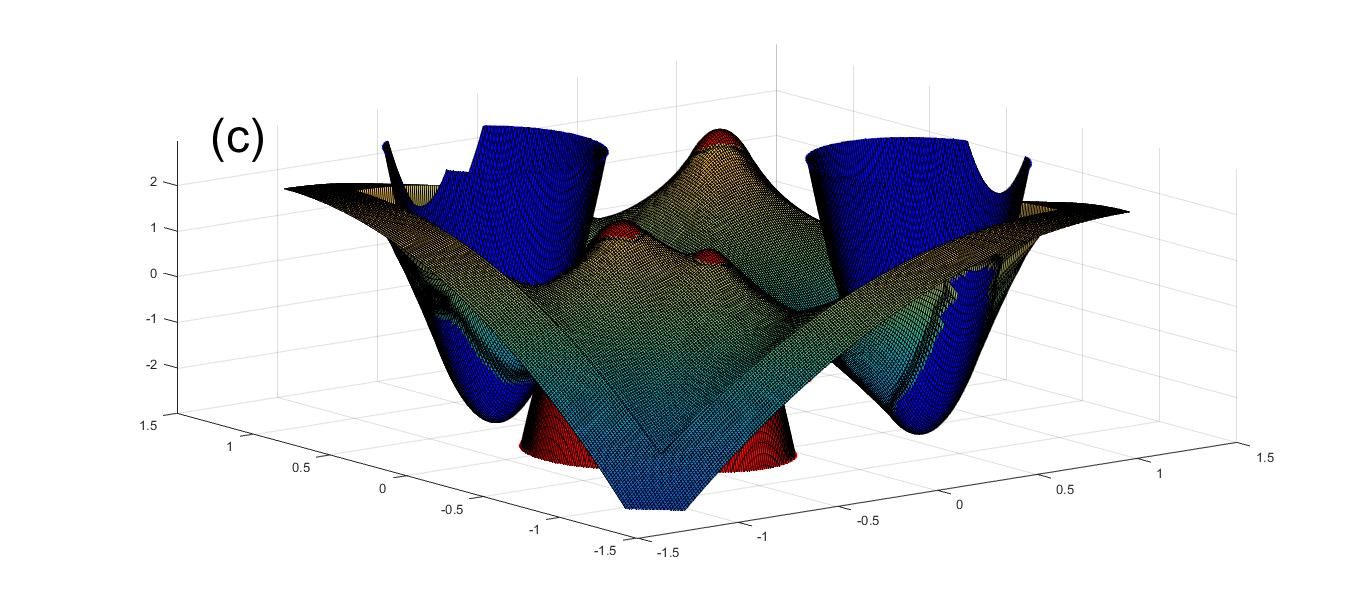}
\caption{Results of tests for $p=10$ and data in (a), (b), (c), respectively.} \label{fig2}
\end{figure}

\medskip

\noindent {\bf The choice of radius and performance.}
We ran tests with the radius $\epsilon$ corresponding to $15$, $10$,
$5$ and $3$ mesh size, for eighteen boundary conditions and obstacle
functions. The table below gathers the information on the
obtained execution time and precision.  
Runtime denotes the average time in seconds it took to
run the experiments for a given radius. Iteration No
denotes the number of times the operator $\bar T$ was applied before obtaining
precision of less than $10^{-3}$. Error 1 is the error measured in
the problem whose known solution is $e^x\sin(y)$ and $p=25$. Error 2 is the error
measured in the problem whose solution is $x^2-y^2-y$ with no
obstacles and $p=2$. 

\bigskip
	
\begin{tabular}{|c|c|c|c|c|c|}
\hline
Radius&$k$ = Points Sampled&Runtime&Iteration No.&Error 1&Error 2\\ \hline
15&709&555&335&$8.62\cdot 10^{-6}$&$8.22\cdot 10^{-11}$\\ \hline
10&317&617&876&$6.17\cdot 10^{-6}$&$8.51\cdot 10^{-11}$\\ \hline
5&81&652&3361&$2.68\cdot 10^{-6}$&$8.68\cdot 10^{-11}$\\ \hline
3&29&540&9255&$3.15\cdot 10^{-7}$&$4.73\cdot 10^{-7}$\\ \hline
\end{tabular}

\bigskip\bigskip

Next we look at how the algorithm performs for different values of
$p$. We ran the algorithm with six different boundary conditions with
no obstacle, one obstacle and two obstacles, each time for 
values $p= 2,3,4,5,10,25,50,100$. The larger the value of $p$, the
faster the algorithm converged as can be seen in the following
table. Each row measures how many iterations it took for the algorithm
to produce a precision of $10^{-3}$ on average over the six boundary
conditions. 
	
\bigskip

\begin{center}
\begin{tabular}{|c|c|c|c|c|c|c|c|}
\hline
$p$&3&4&5&10&25&50&100\\ \hline
No Obstacle&5180&4806&4569&3790&3003&2707&209\\ \hline
One Obstacle&1637&1392&1249&975&825&777&166\\ \hline
Two Obstacles&1842&1933&1366&1108&992&967&178\\ \hline
\end{tabular}
\end{center}

\end{document}